\documentclass[11pt]{article}
\usepackage{latexsym}
\usepackage{amsmath}
\usepackage{amsthm}
\usepackage{amssymb}

\newtheorem{lemma}{Lemma}
\newtheorem{corollary}[lemma]{Corollary}
\newtheorem{theorem}[lemma]{Theorem}
\newtheorem{proposition}[lemma]{Proposition}

\newcommand{\meet}{\wedge}
\newcommand{\join}{\vee}
\newcommand{\fracp}[1]{\{ #1 \}}

\newcommand{\journal}[6]{{\sc #1,} #2, {\it #3} {\bf #4} (#5), #6.}
\newcommand{\preprint}[3]{{\sc #1,} #2, preprint #3.}

\begin{document}
%\setstretch{2}
\title{A geometric approach to acyclic orientations}
\author{Richard Ehrenborg and Michael Slone}

\date{}

\maketitle

\begin{abstract}
The set of acyclic orientations of a connected graph with a given sink
has a natural poset structure. We give a geometric proof of a result
of Jim Propp: this poset is the disjoint union of distributive lattices.
\end{abstract}

Let $G$ be a connected graph on the vertex set
$[\underline{n}] = \{0\} \cup [n]$,
where $[n]$ denotes the set $\{1, \ldots, n\}$.
Let $P$ denote the collection of acyclic orientations of $G$, and 
let $P_{0}$ denote the collection of acyclic orientations of $G$
with $0$ as a sink.
If $\Omega$ is an orientation in $P$ with the vertex $i$
as a source, we can 
obtain a new orientation $\Omega'$ with $i$ as a sink by \emph{firing} the
vertex~$i$,
reorienting all the edges adjacent to $i$ towards $i$.
The orientations $\Omega$ and~$\Omega'$ agree away from $i$.

A \emph{firing sequence} from $\Omega$ to $\Omega'$
in $P$ consists of a sequence
$\Omega=\Omega_1,\dots,\Omega_{m+1}=\Omega'$ of orientations and a
function $F : [m] \longrightarrow [\underline{n}]$
such that for each $i\in [m]$,
the  orientation~$\Omega_{i+1}$ is obtained from $\Omega_i$
by firing the vertex $F(i)$.
We will abuse language by calling $F$ itself a firing sequence.
We make $P$ into a preorder by writing $\Omega\le\Omega'$ if and only if
there is a firing sequence from $\Omega$ to $\Omega'$.  From
the definition it is clear that $P$
is reflexive and transitive.
While $P$ is only a preorder, $P_{0}$ is a poset.  By finiteness, 
antisymmetry can be verified by showing that firing sequences in $P_{0}$
cannot be arbitrarily long.  This is a consequence of the fact that
neighbors of the distinguished sink $0$ cannot fire.
The proof depends on the
following lemma.

\begin{lemma}
Let $F : [m] \longrightarrow [n]$
be a firing sequence for the graph $G$.
If $i$ and $j$
are adjacent vertices in~$G$, then
$$
  |F^{-1}(i)| \le |F^{-1}(j)| + 1.
$$
\end{lemma}
\begin{proof}
A vertex can fire only if it is a source.
Firing the vertex $i$ reverses the orientation
of its edge to the vertex~$j$.  Hence the vertex~$i$
cannot fire again until the orientation is again
reversed, which can only happen by firing $j$.
\end{proof}

As a corollary, firing sequences have bounded length, implying that $P_{0}$ 
is a poset.

\begin{corollary}
The preorder $P_{0}$ of acyclic orientations with
a distinguished sink is a poset.
\end{corollary}
\begin{proof}
Let $F : [m] \longrightarrow [n]$ be a firing sequence.  By 
iterating the lemma, $|F^{-1}(i)| \le d(0, i) - 1$, so
$$
  m =   \sum_{i\in [n]} |F^{-1}(i)| 
    \le \sum_{i\in [n]} (d(0, i) - 1) .
$$
Hence firing sequences cannot be arbitrarily long, implying that $P_{0}$ is
antisymmetric.
\end{proof}

For a real number $a$, let
$\lfloor a \rfloor$ denote the largest integer
less than or equal to $a$. 
Similarly,
let
$\lceil a \rceil$ denote the least integer
greater than or equal to $a$. 
Finally, let $\fracp{a}$ denote the fractional part of the real number $a$,
that is, $\fracp{a} = a - \lfloor a \rfloor$.
(It will be clear from the context if $\fracp{a}$ denotes the fractional
part or the singleton set.)
Observe that the range of the function $x \longmapsto \fracp{x}$
is the half open interval~$[0,1)$.

Let $\widetilde{\mathcal{H}} = \widetilde{\mathcal{H}}(G)$ be
the \emph{periodic graphic arrangement} of the graph $G$,
that is, 
$\widetilde{\mathcal{H}}$ is the collection of all hyperplanes
of the form
$$
     x_{i} = x_{j} + k  , 
$$
where $ij$ is an edge in the graph $G$ and $k$ is an integer.
This hyperplane arrangement 
cuts~$\mathbb{R}^{n+1}$ into open regions.
Note that each region is translation-invariant
in the direction $(1, \ldots, 1)$.
Let $C$ denote the complement of $\widetilde{\mathcal{H}}$,
that is,
$$
    C
  =
    \mathbb{R}^{n+1} \setminus \bigcup_{H \in \widetilde{\mathcal{H}}} H  .
$$
Define a map $\varphi : C \longrightarrow P$
from the complement of the
periodic graphic arrangement to the
preorder of acyclic orientations as follows.
For a point $x = (x_{0}, \ldots, x_{n})$ and an edge~$ij$
observe that $\fracp{x_{i}} \neq \fracp{x_{j}}$
since the point does not lie on any hyperplane
of the form $x_{i} = x_{j} + k$.
Hence orient the edge~$ij$ towards $i$ if $\fracp{x_{i}} < \fracp{x_{j}}$
and towards~$j$ if the inequality is reversed.
This defines the orientation $\varphi(x)$.  Also note that this
is an acyclic orientation, since no directed cycles can occur.

Let $H_{0}$ be the coordinate hyperplane
$\{x \in \mathbb{R}^{n+1} \: : \: x_{0} = 0\}$.
The map $\varphi$ sends points of the intersection $C_0 = C \cap H_{0}$
to acyclic orientations in $P_{0}$.

The real line $\mathbb{R}$ is a distributive lattice; meet is
minimum and join is maximum.  Since~$\mathbb{R}^{n+1}$ is a 
product of copies of $\mathbb{R}$, it is also a distributive
lattice, with meet and join given by componentwise minimum
and maximum.  That is,
given two points 
in~$\mathbb{R}^n$, say 
$x = (x_{0}, \ldots, x_{n})$
and
$y = (y_{0}, \ldots, y_{n})$,
their meet and join are given by
$$
  x \meet y = (\min(x_{0},y_{0}), \ldots, \min(x_{n},y_{n}))
$$
and
$$
  x \join y = (\max(x_{0},y_{0}), \ldots, \max(x_{n},y_{n}))
$$
respectively.

\begin{lemma}
Each region $R$ in the complement $C$ of the periodic graphic arrangement
$\widetilde{\mathcal{H}}$
is a distributive sublattice of $\mathbb{R}^{n+1}$.
Hence the intersection $R \cap H_{0}$, which is a region in $C_{0}$,
is also a distributive sublattice of $\mathbb{R}^{n+1}$.
\end{lemma}
\begin{proof}
Since each region $R$ is the intersection of 
slices of the form
$$
    T
  =
    \{ x \in \mathbb{R} \:\: : \:\: x_i + k < x_j < x_i + k + 1 \},
$$
it is enough to prove that each slice is a sublattice of $\mathbb{R}^{n+1}$.
Let $x$ and $y$ be two points in the slice $T$.
Then
$\min(x_i,y_i) + k   =    \min(x_i + k, y_i + k)
                     <    \min(x_j, y_j)
                     <    \min(x_i + k+1, y_i + k+1)
                     =    \min(x_i,y_i) + k+1$,
implying that $x \meet y$ also lies in the slice $T$.
A dual argument shows that the slice $T$ is closed under the join operation.
Thus the region $R$ is a sublattice.
Since distributivity is preserved under taking sublattices,
it follows that $R$ is a distributive sublattice of $\mathbb{R}^{n+1}$.
\end{proof}

In the remainder of this paper we let $R$ be a region
in $C_{0}$.

\begin{lemma}
Consider the restriction $\varphi|_{R}$ of the map $\varphi$
to the region~$R$.
The inverse image of an acyclic orientation in $P_{0}$ is
of the form:
$$
  R \cap \left(\{0\} \times \prod_{i=1}^{n} [a_{i},a_{i}+1) \right), 
$$
where each $a_{i}$ is an integer. That is, the inverse image
of an orientation is the intersection of the region
$R$ with a half-open lattice
cube.
Hence the inverse image is a sublattice of~$\mathbb{R}^{n+1}$.
\label{lemma_inverse_image}
\end{lemma}
\begin{proof}
Assume that $x$ and $y$ lie in the region $R$.
Define the integers $a_{i}$ and $b_{i}$ by
$a_{i} = \lfloor x_{i} \rfloor$
and
$b_{i} = \lfloor y_{i} \rfloor$.
Hence the coordinate $x_{i}$ lies in the half-open
interval $[a_{i},a_{i}+1)$
and 
the coordinate $y_{i}$ lies in the half-open
interval $[b_{i},b_{i}+1)$.
Lastly, assume that $\varphi|_{R}$ maps $x$ and~$y$ to the
same acyclic orientation.
The last condition implies that, for every edge $ij$,
$0 \leq x_{i} - a_{i} < x_{j} - a_{j} < 1$ is equivalent to
$0 \leq y_{i} - b_{i} < y_{j} - b_{j} < 1$.
Consider an edge that is directed from $j$ to $i$.
Since $x$ and~$y$ both lie in the region~$R$,
there exists an integer~$k$ such that           
$x_{i} + k < x_{j} < x_{i} + k + 1$ 
and         
$y_{i} + k < y_{j} < y_{i} + k + 1$.
Now we have that        
$a_{j} - a_{i} < x_{j} - x_{i} < k+1$.          
Furthermore, observe that           
$x_{j} - a_{j} - 1 < 0 \leq x_{i} - a_{i}$.     
Hence       
$a_{j} - a_{i} > x_{j} - x_{i} - 1 > k-1$.      
Since $a_{j} - a_{i}$ is an integer, the two    
bounds implies that     
$a_{j} - a_{i} = k$.    
By similar reasoning we obtain that
$b_{j} - b_{i} = k$.    

Hence for every edge $ij$ we know that          
$a_{j} - a_{i} = b_{j} - b_{i}$.    
Since $a_{0} = b_{0} = 0$ and the graph $G$ is         
connected we obtain that $a_{i} = b_{i}$ for all
vertices $i$.
\end{proof} 

\begin{lemma}
The restriction $\varphi|_{R} : R \longrightarrow P_{0}$
is a poset homomorphism, that is, 
for two points $y$ and $z$ in the region $R$
such that $y \leq z$ the order relation
$\varphi(y) \leq \varphi(z)$ holds.
\end{lemma}
\begin{proof}
Since the region~$R$ is convex, the line segment from $y$ to $z$ is 
contained in~$R$. 
Let a point $x$ move continuously from $y$ to $z$ along this line
segment
and consider what happens with the associated
acyclic orientations $\varphi(x)$. Note that each coordinate
$x_{i}$ is non-decreasing.
When the point $x$ crosses a hyperplane of the form $x_{i} = p$
where $p$ is an integer, observe that the value $\fracp{x_{i}}$
approaches $1$ and then jumps down to $0$. Hence the vertex~$i$ 
switches from being a source to being a sink, that is, the
vertex~$i$ fires.

Observe that two adjacent nodes $i$ and $j$ cannot fire at the same
time, since the intersection of the two hyperplanes
$x_{i} = p$ and $x_{j} = q$ is contained in the hyperplane
$x_{i} = x_{j} + (p-q)$ which is not in the region $R$.

Hence we obtain a firing sequence from the acyclic orientation
$\varphi(y)$ to $\varphi(z)$,
proving that $\varphi(y) \leq \varphi(z)$.
\end{proof}

\begin{lemma}
Let $x$ be a point in the region $R$.  Let $\Omega'$ be
an acyclic orientation comparable to
$\Omega = \varphi(x)$ in the poset $P_{0}$.
Then there exists a point $z$ in the region of $R$ as $x$
such that $\varphi(z) = \Omega'$.
\label{lemma_lifting}
\end{lemma}
\begin{proof}
It is enough to prove this for cover relations in the poset $P$.
We begin by considering the case when $\Omega'$ covers
$\Omega$ in $P$.  Thus $\Omega'$ is obtained from $\Omega$ by firing
a vertex $i$.

First pick a positive real number $\lambda$ such that
$\fracp{x_{j}} < 1 - \lambda$ for each nonzero vertex~$j$.
Let $y$ be the point $y = x + \lambda \cdot (0,1, \ldots, 1)$.
Observe that $y$ belongs to the same region $R$ and that
$\varphi$ maps $y$ to the same acyclic orientation as the point $x$.

Since $i$ is a source in $\Omega$, the value
$\fracp{y_{i}}$ is larger than any other 
value 
$\fracp{y_{j}}$ for vertexes~$j$ adjacent to the vertex $i$.
Let $z$ be the point with coordinates
$z_{j} = y_{j}$ for $j \neq i$ and
$z_{i} = \lceil y_{i} \rceil + \lambda/2$.
Observe that moving from $y$ to the point $z$
we do not cross any hyperplanes of the form $x_{i} = x_{j} + k$.
Hence the point $z$ also belongs to region $R$.

However, we did cross a hyperplane of the form $x_{i} = p$,
corresponding to firing the vertex~$i$. Hence we have that
$\varphi(z) = \Omega'$.  Now we can iterate this argument to extend
to the general case when $\Omega < \Omega'$.

The case when $\Omega'$ is covered by $\Omega$ is done
similarly. However this case is easier since one can skip
the middle step of defining the point $y$. Hence this case
is omitted.
\end{proof}

A connected component of a finite poset is a weakly connected
component of its associated comparability graph.
That is, a finite poset is the disjoint union
of its connected components.

\begin{lemma}
Let $Q$ be a connected component
of the poset of acyclic orientations~$P_{0}$.
Then there exists a region $R$ in $C_{0}$ such that
the map $\varphi$ maps $R$ onto the component $Q$.
\label{lemma_lifting_components}
\end{lemma}
\begin{proof}
Let $\Omega$ be an orientation in the component $Q$.
Since $\varphi$ is surjective we can lift~$\Omega$ to a point
$x$ in $C_{0}$. Say that the point $x$ lies in the region $R$.
It is enough to show that every orientation
$\Omega^{\prime}$ in $Q$ can be lifted
to a point in $R$. The two orientations
$\Omega$ and $\Omega^{\prime}$ are related by a sequence in $Q$
of orientations 
$\Omega = \Omega_{1}, \Omega_{2}, \ldots, \Omega_{k} = \Omega^{\prime}$
such that $\Omega_{i}$ and $\Omega_{i+1}$ are comparable.
By iterating 
Lemma~\ref{lemma_lifting}
we obtain points~$x_{i}$ in~$R$
such that $\varphi(x_{i}) = \Omega_{i}$.  In particular, 
$\varphi(x_{k}) = \Omega^{\prime}$.
\end{proof}

\begin{proposition}
Let $Q$ be a connected component of the poset of
acyclic orientations $P_{0}$. Then the component $Q$ as a poset is a lattice.
Moreover, let $R$ be a region of $C_{0}$ that maps onto $Q$ by $\varphi$.
Then the poset map $\varphi|_{R} : R \longrightarrow Q$ is
a lattice homomorphism.
\end{proposition}
\begin{proof}
The previous discussion showed that we can lift the component
$Q$ to a region~$R$. Consider two acyclic orientations
$\Omega$ and $\Omega'$. We can lift them to two points~$x$ and~$y$
in~$R$, that is,
$\varphi(x) = \Omega$
and
$\varphi(y) = \Omega'$.
Since $\varphi|_{R}$ is a poset map we obtain that
$\varphi(x \meet y)$ is a lower bound for
$\Omega$ and $\Omega'$. It remains to show that the lower bound
is unique. 

Assume that $\Omega''$ is a lower bound of $\Omega$ and $\Omega'$.
By Lemma~\ref{lemma_lifting}
we can lift $\Omega''$ to an element~$z$ in $R$ such
that $z \leq x$.
Similarly,
we can lift $\Omega''$ to an element $w$ in $R$ such
that $w \leq y$.
That is we have that $\varphi(z) = \varphi(w) = \Omega''$.
Now by Lemma~\ref{lemma_inverse_image}
we have that $\varphi(z \meet w) = \Omega''$.
But since $z \meet w$ is a lower bound of both $x$ and $y$
we have that $z \meet w \leq x \meet y$.
Now applying $\varphi$ we obtain that
$\varphi(x \meet y)$ is the greatest lower bound, proving that
the meet is well-defined.
A dual argument shows that the join is well-defined, hence
$Q$ is a lattice.

Finally, we have to show that $\varphi|_{R}$ is
a lattice homomorphism.
Let $x$ and $y$ be two points in the region $R$.
By Lemma~\ref{lemma_lifting} we can lift
the inequality $\varphi(x) \meet \varphi(y) \leq \varphi(x)$
to obtain a point $z$ in $R$ such that
$z \leq x$ and $\varphi(z) = \varphi(x) \meet \varphi(y)$.
Similarly, we can lift
the inequality $\varphi(x) \meet \varphi(y) \leq \varphi(y)$
to obtain a point $w$ in $R$ such that
$w \leq y$ and $\varphi(w) = \varphi(x) \meet \varphi(y)$.
By Lemma~\ref{lemma_inverse_image} we know that
$\varphi(z \meet w) = \varphi(x) \meet \varphi(y)$.
But~$z \meet w$ is a lower bound of both $x$ and~$y$,
so
$\varphi(x) \meet \varphi(y) = \varphi(z \meet w) \leq \varphi(x \meet y)$.
But since 
$\varphi(x \meet y)$ is a lower bound of both
$\varphi(x)$ and $\varphi(y)$ we have
$\varphi(x \meet y) \leq \varphi(x) \meet \varphi(y)$.
Thus the map $\varphi|_{R}$ preserves the meet operation.
The dual argument proves that 
$\varphi|_{R}$ preserves the join operation,
proving that it is a lattice homomorphism.
\end{proof}

Combining these results we can now prove the result of
Propp~\cite{Propp}.

\begin{theorem}
Each connected component of the poset of acyclic orientations $P_{0}$
is a distributive lattice.
\end{theorem}
\begin{proof}
It is enough to recall that $\mathbb{R}^{n+1}$ is a distributive lattice
and each region $R$ is a sublattice. Furthermore, the image
under a lattice morphism of a distributive lattice is also distributive.
\end{proof}

Observe that the minimal element in each connected component $Q$
is an acyclic orientation with the unique sink at the vertex $0$.
Greene and Zaslavsky~\cite{Greene_Zaslavsky}
proved that the number of such orientations
is given by
the sign $-1$ to the power one less than the number of vertices
times
the linear coefficient in the chromatic polynomial
of the graph $G$.
Gebhard and Sagan gave several proofs of this result~\cite{Gebhard_Sagan}.
A geometric proof of this result can be found 
in~\cite{Ehrenborg_Readdy_Slone}, 
where the authors
view the graphical hyperplane arrangement on a torus
and count the regions on the torus.

That the connected components are confluent,
that is, each pair of elements has a lower and an upper bound,
can also be shown by analyzing 
chip-firing games~\cite{Bjorner_Lovasz_Shor}.
Is there a geometric way to prove
the confluency of chip-firing?
More discussions relating these distributive lattice
with chip-firing can be found in~\cite{Latapy_Magnien,Latapy_Phan}.

\section*{Acknowledgments}

The authors were partially supported by
National Security Agency grant H98230-06-1-0072.
The authors thank Andrew Klapper and Margaret Readdy
for their comments
on an earlier version of this paper.

\bigskip

{\em R.\ Ehrenborg and M.\ Slone,
Department of Mathematics,
University of Kentucky,
Lexington, KY 40506-0027,}
\{{\tt jrge},{\tt mslone}\}{\tt @ms.uky.edu}

\end{document}